\documentclass [11pt]{article}
\usepackage{amsthm}
\usepackage{amsmath}
\usepackage{amssymb}
\usepackage{enumerate}
\usepackage{epsfig}
\usepackage{srcltx}
\usepackage{graphics}

\newcounter{contador}
\newcounter{contador2}

\newtheorem{theorem}[contador]{Theorem}
\newtheorem{lemma}[contador]{Lemma}

\newcommand{\C}{{\mathbb C}}
\newcommand{\N}{{\mathbb N}}

\newcommand{\pr}{{P}{\mathbb C^2}}
\newcommand{\pru}{{P}{\mathbb C^1}}

\newcommand{\Z}{{\mathbb Z}}

\newcommand{\homog}{[x_0:x_1:x_2]}

\title{Dynamical Classification of a Family of Birational Maps of $\C^2$ via Algebraic Entropy\footnote{{\bf Acknowledgements}.
The first author is supported by Ministry of Economy, Industry and Competitiveness of the Spanish Government through grants MINECO/FEDER MTM2016-$77278$ and also supported by the grant $2014$-SGR-$568$ from AGAUR, Generalitat de Catalunya. The GSD-UAB Group is supported by the Government of Catalonia through the SGR program. It is also supported by MCYT through the grant MTM$2008-03437.$}}

\author{Anna Cima$^{(1)}$ and Sundus Zafar$^{(1)}$
  \\*[.1truecm]
{\small \textsl{$^{(1)}$ Dept. de Matem\`{a}tiques, Facultat de
Ci\`{e}ncies,}}
\\*[-.25truecm] {\small \textsl{Universitat Aut\`{o}noma de Barcelona,}}
\\*[-.25truecm] {\small \textsl{08193 Bellaterra, Barcelona, Spain}}
\\*[-.25truecm] {\small \textsl{cima@mat.uab.cat,
sundus@mat.uab.cat}}}

\date{}

\begin{document}
\maketitle



\begin{abstract}
This work dynamically classifies a $9-$parametric family of
birational maps $f: \C^2 \to \C^2.$  From the sequence of the
degrees $d_n$ of the iterates of $f,$ we find the dynamical degree
$\delta(f)$ of $f$. We identify when $d_n$ grows periodically,
linearly, quadratically or exponentially. The considered family
includes the birational maps studied by Bedford and Kim in
\cite{BK2} as one of its subfamilies.
\end{abstract}


\noindent {\sl  Mathematics Subject Classification 2010:} 14E05,
26C15, 34K19, 37B40, 37C15, 39A23, 39A45.

\noindent {\sl Keywords:} Birational maps, Algebraic entropy, First
Integrals, Fibrations, Blowing-up, Integrability, Periodicity, Chaos.

\section{Introduction}

In this work we
consider the family of fractional maps $f:\C^2 \to \C^2$ of the
form:
\begin{equation}
\label{eq1}
 f(x,y) = \left( {\alpha _0} + {\alpha _1}x + {\alpha
_2}y,\frac{{\beta _0} + {\beta _1}x + {\beta _2}y}{{\gamma _0} +
{\gamma _1}x + {\gamma _2}y} \right),
\end{equation}
where the parameters are complex numbers.

This family of maps can be extended to the projective plane $\pr$ by
considering the embedding $(x_1,x_2)\in\C^2 \mapsto [1:x_1:x_2] \in
{\pr}$ into projective space. The induced map $F:{\pr} \to {\pr}$
has three components $F_i[x_0:x_1:x_2]\,,\,i=1,2,3$ which are
homogeneous polynomials of degree two. For general values of the
parameters the three components don't have a common factor: we say
that these maps have degree two. Similarly we can define the degree
of ${F^n} = F \circ \cdot \cdot \cdot  \circ F$ for each $n\in\N.$
It can be seen that if $f(x_1,x_2)$ is a birational map, then the
sequence of its degrees satisfies a homogeneous linear recurrence
with constant coefficients (see \cite{DF} for instance or Section
3). This is governed by the characteristic polynomial
$\mathcal{X}(x)$of a certain matrix associated to $F.$ The other
information we get from $\mathcal{X}(x)$ is the \textit{dynamical
degree} $\delta(F)$ which is it's largest real root, and is defined
as
\begin{equation}\label{eq.1}
\delta(F): = \mathop {\lim }\limits_{n \to \infty } {\left( {\deg
({F^n})} \right)^{\frac{1}{n}}},
\end{equation}
see \cite{BK1,BK2,BK3,BV,JD,DF}. The logarithm of this quantity has
been called the \textit{algebraic entropy}.

The application of algebraic entropy in the field of dynamical
systems has been growing in recent years, see for instance
\cite{BK1,BK2,BK3,BV,BD,BC,CZ,JD,DF}. On the other hand, the study
of the dynamics generated by birational mappings in the plane is
also a current issue, see for instance \cite{ADMV,BK1,BK2,BK3,JD}
also \cite{CGM2,CGM3,DS,LG,LKP,LGK,LKMR,Ro1,PRo,ZE}.

It is known (see \cite{Yom}) that the algebraic entropy is an upper
bound of the \textit{topological entropy}, which in turn is a
dynamic measure of the complexity of the mapping.

The algebraic entropy for the maps (\ref{eq1}) highly depends on the
choice of parameters. For this reason our study includes all the
possible values of the parameters of $f$ to determine the growth
rate $d_n$ and $\delta(F).$ Therefore the results that we get can be
seen as a dynamical classification of family (\ref{eq1}).
Furthermore, they generalize the results obtained in \cite{BK2},
which the authors consider the subfamily of (\ref{eq1})with ${\alpha
_0}=0,\,{\alpha _1}=0$ and ${\alpha_2}=1.$

Birational mappings $F:{\pr} \to {\pr}$ have an indeterminacy set
${\mathcal{I}(F)}$ of points where $F$ is ill-defined as a
continuous map. This set is given by:
$${\mathcal{I}(F)}=\{\homog\in\pr:F_1\homog=0,F_2\homog=0,F_3\homog=0]\}.$$


On the other hand, if we consider one irreducible component $V$ of
the determinant of the Jacobian of $F$, it is known (see Proposition
$3.3$ in \cite{JD}) that $F(V)$ reduces to a point in
$\mathcal{I}(F^{-1})$. The set of these curves which are sent to a
single point is called the {\it exceptional locus} of $F$ and it is
denoted by $\mathcal{E}(F).$

It is known that the dynamical degree depends on the orbits of the
indeterminacy points of the inverse of $F$ under the action of $F,$
see \cite{DF, FS, SZ}. Indeed, the key point is whether the iterates
of such points coincide with any of the indeterminacy points of $F.$

Generically our map $F$ has three indeterminacy points. The
exceptional locus is formed by three straight lines, each two of
them intersecting on a single indeterminate point of $F$. We call
them \textit{non degenerate mappings}. But there is a subfamily such
that the exceptional locus is formed by only two straight lines. We
call these mappings {\it degenerate mappings} and they are studied
in the paper \cite{CZ1}. Hence we are not going to consider them
here.

To find $\delta(F)$ and $d_n$ and to study its behaviour, we use a
Theorem of Bedford and Kim (see \cite{BK1}) to find the
characteristic polynomial which provides $d_n$.

The results obtained in this paper are the starting point for
studying the dynamic properties of the elements of family
(\ref{eq1}). We can expect certain types of behaviors, for instance,
if we want to know the mappings which are globally periodic, we have
to look at the ones whose sequence of degrees is periodic. To find
the mappings that are integrable (i. e., mappings that preserve the
level curves of some rational function) or mappings that preserve
some fibrations, we have to look for the mappings whose sequence of
degrees $d_n$ grows linearly or quadratically in $n,$ see \cite{DF}.
We can encounter chaos whenever $d_n$ grows exponentially. As a
continuation of this work, in the following articles, "Zero entropy
for some birational maps of $\C^2$", see \cite{CZ0}, and
"Finding invariant fibrations for some birational maps of $\C^2$",
see \cite{CZ1}, we give all the maps of type (\ref{eq1}) which have
zero entropy in the particular cases $\gamma_1=0$ in the first one
and for the degenerate cases in the second one, giving explicitly
the invariants when they exist.

The details of prerequisites and results of this work can be found
in \cite{CZ2}.

The article is organized as follows: The main results are announced
in Section $2.$ The preliminary results which include the basic
settings of the work, some background of birational maps and Picard
group and the structure of the orbits'lists are introduced in
Section $3.$ In Section $4$ we give the proof of the results.
Finally in Section $5$ we present the rest of the zero entropy
mappings which are not included in the two mentioned papers
\cite{CZ0} and \cite{CZ1}.

\section{Main results}

The results that we find are presented in the following theorems
$1,2,3,4$ and $5$. In all of them we consider that the coefficients
of the map $f$ are such that $f$ is a birational map and $F$ has
degree two (see Lemma \ref{conditions}).

Consider the family of fractional maps $f:\C^2 \to \C^2$ :
$$
 f(x,y) = \left( {\alpha _0} + {\alpha _1}x + {\alpha
_2}y,\frac{{\beta _0} + {\beta _1}x + {\beta _2}y}{{\gamma _0} +
{\gamma _1}x + {\gamma _2}y} \right),
$$
where the parameters are complex numbers. We call
\begin{equation}
\label{EQ1} F[{x_0}:{x_1}:{x_2}]
=[F_1\homog:F_2\homog:F_3\homog],\end{equation} the extension of
$f(x,y)$ to the projective plane, where
$$\begin{array}{l}
F_1\homog=x_0({\gamma _0}x_0 + {\gamma _1}x_1 +
{\gamma _2}x_2),\\
F_2\homog=({\alpha _0}x_0 + {\alpha _1}x_1 + {\alpha _2}x_2)({\gamma
_0}x_0 + {\gamma _1}x_1 + {\gamma _2}x_2)
,\\
F_3\homog=x_0({\beta _0}x_0 + {\beta _1}x_1 + {\beta _2}x_2).
\end{array}$$

The indeterminacy set of  $F[x_0:x_1:x_2]$ is
${\mathcal{I}(F)}\,=\,\{O_0,O_1,O_2\},$ with
\begin{equation}\label{Os}
O_0=[(\beta\gamma)_{12}:(\beta\gamma)_{20}:(\beta\gamma)_{01}],\quad
O_1=[0:\alpha_2:-\alpha_1],\quad O_2=[0:\gamma_2:-\gamma_1],
\end{equation} where
$(\beta\gamma)_{ij}:=\beta_i\gamma_j-\beta_j\gamma_i$ for $i,j\in
\{0,1,2\}.$

By calling $f^{-1}(x,y)$ the inverse of $f(x,y)$ and by
$F^{-1}\homog$ its extension on $\pr,$ also a indeterminacy set
${\mathcal{I}(F^{-1})}$ exists:
${\mathcal{I}(F^{-1})}\,=\,\{A_1,A_2,A_3\},$ with
$$A_0=[0:1:0]\quad , \quad A_1=[0:0:1],$$
\begin{equation}\label{As} A_2=[(\beta\gamma)_{12}\,(\alpha\gamma)_{12}:(\alpha_0\,(\beta\gamma)_{12}-\alpha_1\,(\beta\gamma)_{02}+
\alpha_2\,(\beta\gamma)_{01})\,(\alpha\gamma)_{12}:(\alpha\beta)_{12}\,(\beta\gamma)_{12}].
\end{equation}

We denote by $\delta^{*} = \frac{1+\sqrt{5}}{2}$  the golden mean,
which is the largest root of the polynomial $x^2-x-1.$

\begin{theorem}\label{theo1}
Let $F:\pr \to \pr$ be a birational degree two non degenerate map of
type (\ref{eq1}) and suppose that $\alpha_1, \alpha_2, \gamma_1,
\gamma_2$ are all non zero. Then either,
\begin{itemize}
\item [(i)] If it exists $p\in\N$ such that ${F^p}({A_2}) = {O_0},$ then the
characteristic polynomial associated with $F$ is
 $$\mathcal{W}_p = {x^{p + 2}} - 2{x^{p + 1}} + x - 1,$$  $\delta(F)$ is given by the largest root of the
polynomial $\mathcal{W}_p$ and $d_n$ grows quadratically.
 \item [(ii)] If
no such $p$ exists then $\delta(F) = 2$ and $d_n$ grows
quadratically.
\end{itemize}
\end{theorem}

Notice that Theorem \ref{theo1} says us that family (\ref{eq1})
generically has dynamical degree equal $2$.

\begin{theorem}\label{theo2}
Let $F:\pr \to \pr$ be a birational degree two non degenerate map of
type (\ref{eq1}) and suppose that $\gamma_1 = 0$. Then $\gamma_2,
\alpha_1, \beta_1$ are non zero and the following hold:
\begin{enumerate}
\item Assume that $\alpha_2 = 0$ and let $\tilde{F}$ be the
extension of $F$ after blowing-up the points $A_0,A_1.$  If
$\tilde{F}^{p}(A_2) = O_0$ for some $p \in \N$ then the
characteristic polynomial associated with $F$ is given by
\begin{equation*}
\mathcal{X}_p = (x^{p+1}+1)(x-1)^2(x+1),
 \end{equation*}
and the sequence of degrees $d_n$ of $F$ is periodic with period
$2p+2.$ If no such $p$ exists then the characteristic polynomial
associated with $F$ is
$$\mathcal{X}\,=\,(x-1)^2\,(x+1),$$ and the sequence of degrees $d_n$ grows
linearly.

\item Assume that $\alpha_2 \ne 0$ and let $\tilde{F}$ be the
induced map after blowing up the point $A_0.$ Then the following
hold:
\begin{itemize}
    \item If $\tilde{F}^{p}(A_2) = O_0$ for some $p \in \N$ and $\tilde{F}^{2k}(A_1) \ne
O_1$ for all $k \in  \N$ then the characteristic polynomial
associated with $F$ is given by
\begin{equation*}
\mathcal{X}_p = x^{p+1}(x^2-x-1)+x^2,
 \end{equation*}
 and
 \begin{itemize}
    \item for $p = 0,\,p=1$ the sequence of degrees $d_{n}$ is bounded,
    \item for $p = 2$ the sequence of degrees $d_{n}$ grows linearly,
    \item for $p > 2$ the sequence of degrees $d_{n}$ grows exponentially.
\end{itemize}

 \item Assume that $\tilde{F}^{2k}(A_1) = O_1$ for some $k \in \N.$ Let
$\tilde{F}_1$ be the induced map after we blow-up the points
$A_0,A_1,\tilde{F}(A_1),\ldots ,\tilde{F}^{2k}(A_1)=O_1.$  If
$\tilde{F}_1^{p}(A_2) \ne O_0$ for all $p\in\N,$ then the
characteristic polynomial associated with $F$ is given by

\begin{equation*}\label{sik}
\mathcal{X}_k = x^{2k+1}(x^2-x-1)+1,
\end{equation*}
and the sequence of degrees grows exponentially. Furthermore
$\delta(F) \to \delta^{*}$ as $k \to \infty.$

\item If $\tilde{F}^{2k}(A_1) = O_1$ and $\tilde{F}_1^{p}(A_2) = O_0$ for some $p,k \in \N$ then the characteristic polynomial
associated with $F$ is given by
\begin{equation*}
\mathcal{X}_{(k,p)} =
x^{p+1}(x^{2k+3}-x^{2k+2}-x^{2k+1}+1)+x^{2k+3}-x^2-x+1,
 \end{equation*}
and
\begin{itemize}
    \item for $p > \frac{2\,(1+k)}{k}$ the sequence of degrees $d_{n}$ grows exponentially for all $p,\,k \in \N;$
    \item for $(p,k) \in \{(3,\,2),\,(4,\,1)\}$ the sequence of degrees $d_{n}$ is periodic or grows
    quadratically;
    \item for $(p,k) \in \{(0,\,k),\,(1,\,k),\,(2,\,k),\,(3,\,1)\}$ the sequence of degrees $d_{n}$ is periodic.
    \end{itemize}
 \item Assume that $\tilde{F}^{2k}(A_1)\ne O_1$ and $\tilde{F}^p(A_2)\ne O_0$ for all
$k,\,p\in\N.$ Then the characteristic polynomial associated with $F$
is given by $$\mathcal{X}(x) = {x^2} - x - 1,$$ and the sequence of degrees
grows exponentially with $\delta(F)=\delta^{*}.$
\end{itemize}
\end{enumerate}
\end{theorem}

\begin{theorem}\label{theo3}
Let $F:\pr \to \pr$ be a birational degree two non degenerate map of
type (\ref{eq1}) and suppose that $\gamma_2 = 0$. Then $\gamma_1,
\alpha_2, \beta_2$ are non zero and the following hold:
\begin{enumerate}
\item Let $F$ be the map for $\alpha_1 = 0$ and let $\tilde{F}$ be the induced map after blowing up the points $A_0\,,A_1.$ If $\tilde{F}^{p}(A_2) = O_0$ for some $p \in \N$ then the characteristic polynomial associated with $F$ is given by
\begin{equation*}
\mathcal{Y}_p = x^{p+1}(x^3-x-1)+(x^3+x^2-1),
 \end{equation*}
 and
 \begin{itemize}
    \item for $p \in \{0,\,1,\,2,\,3,\,4,\,5\}$ the sequence of degrees $d_{n}$ is
    periodic of period $6,5,8,12,18$ and $30$ respectively;
    \item for $p = 6$ the sequence of degrees $d_{n}$ it grows
    quadratically or it is periodic of period $30$;
    \item for $p > 6$ the sequence of degrees $d_{n}$ grows exponentially.
\end{itemize}
If no such $p$ exists then the characteristic polynomial associated
with $F$ is given by
$$\mathcal{Y}= x^3-x-1,$$ and the sequence of degrees grows exponentially with $\delta(F) = \delta_{*}.$

\item Let $F$ be the map for $\alpha_1 \ne 0$ and let $\tilde{F}$ be the induced map after blowing up the point $A_1.$ Then the following hold:
\begin{itemize}
\item If $\tilde{F}^{p}(A_2) = O_0$ for some $p \in  \N$ then the characteristic polynomial associated with $F$ is given by
\begin{equation*}
\mathcal{Y}_p = x^{p+1}(x^2-x-1)+x^2-1,
 \end{equation*}
and the sequence of degrees has exponential growth rate. Furthermore
 $\delta(F) \to \delta^{*}$ as $p \to \infty.$
 \item If $\tilde{F}^{q}(A_2) = O_1$ for some $q \in  \N$ then the characteristic polynomial associated with $F$ is given by
\begin{equation*}
\mathcal{Y}_q = x^{q+1}(x^2-x-1)+x^2,
 \end{equation*}
 and for $q \geq 2$:
 \begin{itemize}
\item The sequence of the degrees grows linearly when $q=2.$
\item The sequence of the degrees grows exponentially when $q>2.$
\end{itemize}
For $q \in \{0\,,1\}$ there are no such mappings.
\item Assume that $\tilde{F}^{p}(A_2)\ne O_0$ and $\tilde{F}^q(A_2)\ne O_1$ for all
$q,\,p\in\N.$ Then the characteristic polynomial associated with $F$
is given by $$\mathcal{Y}(x) = {x^2} - x - 1,$$ and the sequence of degrees
grows exponentially with $\delta(F)=\delta^{*}.$
 \end{itemize}
\end{enumerate}
\end{theorem}

\begin{theorem}\label{theo4}
Let $F:\pr \to \pr$ be a birational degree two non degenerate map of
type (\ref{eq1}) and suppose that $\gamma_1\ne 0,$
 $\gamma_2\ne 0$ and $\alpha_1\alpha_2=0.$ Then:
\begin{enumerate}
\item Assume that $\alpha_1 = 0$ and let $\tilde{F}$ be the induced map after blowing up the point $A_0.$ Then the following hold:
\begin{itemize}
\item If $\tilde{F}^{p}(A_1) = O_0$ for some $p \in  \N$ and $\tilde{F}^{q}(A_2) \ne O_0$  for all $q \in  \N$ then the characteristic polynomial associated with $F$ is given by
\begin{equation*}
\mathcal{Z}_p = x^{p+1}(x^2-x-1)+x^2,
 \end{equation*}
 and for $p \geq 2$:
 \begin{itemize}
\item The sequence of the degrees grows linearly for $p=2.$
\item The sequence of the degrees grows exponentially for $p>2$.
\end{itemize}
For $p \in \{0\,,1\}$ there are no such mappings.
\item If $\tilde{F}^{q}(A_2) = O_0$ for some $q \in  \N$ and $\tilde{F}^{p}(A_1) \ne O_0$ for all $p \in  \N$ then the characteristic polynomial associated with $F$ is given by
\begin{equation*}
\mathcal{Z}_q = x^{q+1}(x^2-x-1)+x^2-1,
 \end{equation*}
 and the sequence of degrees has exponential growth rate. Furthermore  $\delta(F) \to \delta^{*}$ as $q \to \infty.$
\item If $\tilde{F}^{p}(A_1)= O_0$ and $\tilde{F}^q(A_2)= O_0$ for some
$p,\,q\in\N,$ then $p \ne q$ and
\begin{itemize}
\item for $p > q$ the characteristic polynomial associated with $F$ is $\mathcal{Z}_q.$
\item for $p < q$ the characteristic polynomial associated with $F$ is $\mathcal{Z}_p.$
\end{itemize}

\item If $\tilde{F}^{p}(A_1)\ne O_0$ and $\tilde{F}^q(A_2)\ne O_0$ for all
$p,\,q\in\N$ then the characteristic polynomial associated with $F$
is given by $$\mathcal{Z}(x) = {x^2} - x - 1,$$ and the sequence of degrees
grows exponentially with $\delta(F) = \delta^{*}.$
\end{itemize}

\item Assume $\alpha_2= 0$ and let $\tilde{F}$ be the induced map after blowing up the point $A_1.$
If there exists some $p \in  \N$ such that $\tilde{F}^{p}(A_2) =
O_0,$ then the characteristic polynomial associated with $F$ is
given by
\begin{equation*}
\mathcal{Z}_p = (x^{p+1}+1)(x-1)^2,
 \end{equation*}
and for all $p \in \N$ the sequence of degrees $d_n$ grows linearly.
If no such $p$ exists then the characteristic polynomial associated
with $F$ is given by $\mathcal{Z}= (x-1)^2,$ and $d_n$ grows
linearly.
\end{enumerate}
\end{theorem}

\section{Preliminary results}
\subsection{Settings}

Consider
$$
 f(x,y) = \left( {\alpha _0} + {\alpha _1}x + {\alpha
_2}y,\frac{{\beta _0} + {\beta _1}x + {\beta _2}y}{{\gamma _0} +
{\gamma _1}x + {\gamma _2}y} \right).$$ The exceptional locus of
$F[x_0:x_1:x_2]$ is ${\mathcal{E}(F)}\,=\,\{S_0,S_1,S_2\},$ where
\begin{equation*}
S_0=\{x_0=0\},\quad S_1=\{\gamma_0 x_0+\gamma_1 x_1+\gamma_2 x_2=0\},
\end{equation*}
\begin{equation*}
 S_2=\{\left(\alpha_1
(\beta\gamma)_{02}-\alpha_2 (\beta\gamma)_{01}\right)\,x_0+\alpha_1
(\beta\gamma)_{12} x_1+\alpha_2 (\beta\gamma)_{12} x_2=0\},
\end{equation*}
and the exceptional locus of $F^{-1}[x_0:x_1:x_2]$ is
${\mathcal{E}(F^{-1})}\,=\,\{T_0,T_1,T_2\},$ where
$$\begin{array}{l}
T_0=\left\{\left(\gamma_0 (\alpha\beta)_{12}-\gamma_1 (\alpha\beta)_{02}+\gamma_2 (\alpha\beta)_{01}\right)x_0-(\beta\gamma)_{12} x_1=0\right\},\\
T_1=\{(\alpha\beta)_{12} x_0-(\alpha\gamma)_{12} x_2=0\},\quad\quad T_2=\{x_0=0\}.
\end{array}$$

It is easy to see that $F$ maps each $S_i$ to $A_i$  where the
$A_i's$ are defined in (\ref{As}) and that the inverse of $F$ maps
$T_i$ to $O_i$ for $i\in \{0,1,2\},$ see (\ref{Os}). To specify this
behaviour we write $F:S_i\twoheadrightarrow A_i$ (also
$F^{-1}:T_i\twoheadrightarrow O_i$).

We are interested in the mappings (\ref{eq1}) when they are
birational maps which are not degree one maps. Next lemma informs
about the set of parameters which are available in this study. Also
the degenerate case and the non degenerate cases are distinguished. We recognize the non degenerate case when $F$ has three distinct exceptional curves.
When $f$ has two exceptional curves of such type, then we are in degenerate case.

Recall that a birational map is a map $ f:\C^2\rightarrow \C^2$ with
rational components such that there exists an algebraic curve $V$ and
another rational map $g$ such that $f\circ g=g\circ f=id$ in
$\C^2\setminus V.$

\begin{lemma}\label{conditions}
Consider the mappings $$f(x_1,x_2)=\left( {\alpha _0} + {\alpha
_1}x_1 + {\alpha _2}x_2,\frac{{\beta _0} + {\beta _1}x_1 + {\beta
_2}x_2}{{\gamma _0} + {\gamma _1}x_1 + {\gamma _2}x_2} \right),\,(\gamma_1,\gamma_2) \neq (0,0)\neq\,(\alpha_1,\alpha_2).$$

Then: \begin{itemize}
\item [(a)] The mapping $f$ is birational if
and only if the vectors
$(\beta_0,\beta_1,\beta_2),\,\,(\gamma_0,\gamma_1,\gamma_2)$ are
linearly independent and
$((\alpha\beta)_{12},(\alpha\gamma)_{12})\ne (0,0),\,((\alpha\gamma)_{12},(\beta\gamma)_{12})\ne (0,0),$ and either
$((\alpha\beta)_{12},(\beta\gamma)_{12})\ne (0,0)$ or $(\beta_1,\,\beta_2) = (0,0).$

\item [(b)] The mapping $f$ is degenerate if and only if $(\beta\gamma)_{12}=0$ or $(\alpha\gamma)_{12}=0.$
\end{itemize}\end{lemma}

\begin{proof}
The conditions in {\it (a)} are necessary for $f$ to be invertible
as if the vectors $(\beta_0,\beta_1,\beta_2),$
$(\gamma_0,\gamma_1,\gamma_2)$ are linearly dependent then the
second component of $f$ is a constant, also if
$((\alpha\beta)_{12},$ $(\alpha\gamma)_{12})= (0,0)$ or
$((\alpha\gamma)_{12},(\beta\gamma)_{12})= (0,0)$ then $f$ only
depends on $\alpha_1\,x_1+\alpha_2\,x_2$ or on
$\gamma_1x_1+\gamma_2x_2.$ If
$((\alpha\beta)_{12},(\beta\gamma)_{12})= (0,0)$ and
$(\beta_1,\,\beta_2) \neq (0,0)$ then $f$ only depends on
$\beta_1x_1+\beta_2x_2$.

Now assume that conditions $(a)$ are satisfied. Then the inverse of
$f$ which formally is
$$f^{-1}(x,y)=\left(\frac{-(\alpha\beta)_{02}+\beta_2 x+(\alpha\gamma)_{02}y-\gamma_2 x
y}{(\alpha\beta)_{12}-(\alpha\gamma)_{12}y},\frac{(\alpha\beta)_{01}-\beta_1
x+(\alpha\gamma)_{10}y+\gamma_1 x
y}{(\alpha\beta)_{12}-(\alpha\gamma)_{12}y}\right),$$ is well
defined. Furthermore the numerators of the determinants of the Jacobian of $f$ and
$f^{-1}$ are
\begin{equation}\label{detf}
\alpha_1(\beta\gamma)_{02}-\alpha_2(\beta\gamma)_{01}+\alpha_1(\beta\gamma)_{12}x+\alpha_2(\beta\gamma)_{12}y
\end{equation}
and
\begin{equation}\label{detfinv}
\alpha_0(\beta\gamma)_{12}-\alpha_1(\beta\gamma)_{02}+\alpha_2(\beta\gamma)_{01}-(\beta\gamma)_{12}y,
\end{equation}
respectively. It is easily seen that conditions $(a)$ imply that
both (\ref{detf}) and (\ref{detfinv}) are not identically zero. Hence, $f\circ
f^{-1}=f^{-1}\circ f=id$ in $\C^2\setminus V,$ where $V$ is the
algebraic curve determined by the common zeros of (\ref{detf}) and
(\ref{detfinv}).

 To see {\it (b)} we know that since $S_i$ maps to $A_i,$
this implies that the points $A_0,A_1,A_2$ are not all distinct. Since
$A_0\ne A_1$ we have two possibilities: $A_0=A_2$ or $A_1=A_2.$
Condition $A_0=A_2$ writes as
$(\beta\gamma)_{12}\,(\alpha\gamma)_{12}=0$ and
$(\alpha\beta)_{12}(\beta\gamma)_{12}=0.$ From {\it (a)}, the vector
$((\alpha\beta)_{12},(\alpha\gamma)_{12})\ne (0,0).$ Hence
$(\beta\gamma)_{12}$ must be zero. In a similar way it is seen that
$A_1=A_2$ if and only if $(\alpha\gamma)_{12}=0.$
\end{proof}


\subsection{Birational mappings and Picard group}
Given the birational map $f$ let $F[x_0:x_1:x_2]$ be the extension
of $f(x_1,x_2)$ at $\pr$ and consider ${\mathcal{I}(F)}$  and
$\mathcal{E}(F).$ To get rid of indeterminacies we do a series of
blowups. More precisely, if $F^k(A_i)=O_j$ we perform the blowingup
at the points $A_i,F(A_i),\ldots ,F^k(A_i)=O_j.$

Given a point $p\in\C^2,$ let $(X,\pi),$ be the blowing-up of $\C^2$
at the point $p.$ Then,
$$\pi^{-1}p=\pi^{-1}(0,0)=\{\left((0,0),[u:v]\right)\}:=E_p\simeq\pru$$
and if $q=(x,y)\ne (0,0),$ then
$$\pi^{-1}q=\pi^{-1}(x,y)=\left((x,y),[x:y]\right)\in X.$$
Given the point $\left((0,0),[u:v]\right)\in E_{p}$ (resp.
$\left((x,y),[x:y]\right)$) we are going to represent it by
$[u:v]_{E_p}$ (resp. by $(x,y)\in\C^2$ or by $[1:x:y]\in\pr$ if it
is convenient). After every blow up we get a new expanded space $X$
and the induced map $\tilde{F}: X \to X.$ Hence in this work we deal
with complex manifolds $X$ obtained after performing a finite
sequence of blow-ups. Indeterminacy sets and exceptional locus can
also be defined if we consider meromorphic functions defined on
complex manifolds. If $X$ is a complex manifold we are going to
consider the Picard group of $X,$ denoted by $\mathcal{P}ic(X).$
Then $\mathcal{P}ic(\pr)$ is generated by the class of $L,$ where
$L$ is a generic line in $\pr.$ As usual, given a curve $C$ on
$\C^2,$ the {\it strict transform} of $C$ is the adherence of
${\pi^{-1}}{(C\setminus\{p\})},$ in the Zariski topology, and we
denote it by $\hat C.$ If the base points of the blow-ups are
$\{p_1, p_2,\ldots ,p_k\}\subset \pr$ and $E_i:=\pi^{-1}\{p_i\}$
then it is known that $\mathcal{P}ic(X)$ is generated by $\{\hat L,
E_1, E_2, \ldots ,E_k\},$ where $L$ is a generic line in $\pr$ (see
\cite{BK1, BK2}). Furthermore $\pi:X\longrightarrow \pr$ induces a
morphism of groups $\pi^*:\mathcal{P}ic(\pr)\longrightarrow
\mathcal{P}ic(X),$ with the property that for any complex curve
$C\subset \pr,$
\begin{equation}\label{clau}
\pi^*(C)=\hat C+\sum m_i\,E_i,
\end{equation}
where $m_i$ is the algebraic multiplicity of $C$ at $p_i.$

On the other hand, if $F$ is a birational map defined on $\pr,$ then
there is a natural extension of $F$ on $X,$ which we denote by
$\tilde F.$ And $\tilde F$ induces a  morphism of groups,
 $\tilde F^*:\mathcal{P}ic(X)\rightarrow \mathcal{P}ic(X)$ just by
 taking classes of preimages. The interesting thing here
 is that
 $$\tilde F^*(\hat L)\,=\,d\,\hat L\,+\,\sum_{i=1}^k c_i\,E_i\quad ,
 \quad
 c_i\in\Z$$
 where $d$ is the degree  of $F.$ By iterating $F,$ we get the
 corresponding formula by changing $F$ by $F^n$ and $d$ by $d_n.$
In order to deduce the behavior of the sequence $d_n$ it is
convenient to deal with maps $\tilde F$ such that
\begin{equation}\label{AS}
(\tilde F^{n})^*=(\tilde F^*)^n.
\end{equation}
Maps $\tilde F$  satisfying condition (\ref{AS}) are called {\it
Algebraically Stable maps} (AS for short), (see \cite{DF}).

In order to get AS maps we will use the following useful result
showed by Fornaess and Sibony in \cite{FS} (see also Theorem 1.14)
of \cite{DF}:
\begin{equation}\label{condicio}
{\text{The map $\tilde{F}$ is AS if and only if for every
exceptional curve}\,\, C\,\, \text{and all} \,\, n\ge 0\,,\,\tilde
F^n(C)\notin {\mathcal{I}}(\tilde F).}
\end{equation}

It is known (see Theorem 0.1 of \cite{DF}) that one can always
arrange for a birational map to be AS considering an extension of
$f.$ If it is the case and we call $\mathcal
{X}(x)=x^k+\sum_{i=0}^{k-1} c_i\,x^i$ the characteristic polynomial
of $A:=(\tilde F^*),$ then since $\mathcal {X}(A)=0$ and $d_i$ is
the $(1,1)$ term of $A^i$ we get that
$$d_{k}=-(c_0+c_1d_1+c_2d_2+\cdots + c_{k-1}d_{k-1}),$$
i. e., the sequence $d_n$ satisfies a homogeneous linear recurrence
with constant coefficients. The dynamical degree is then the largest real root of $\mathcal {X}(x).$




The following result is useful in our work. It is a direct
consequence of Theorem 0.2 of \cite{DF}. Given a birational map $F$
of $\pr,$ let $\tilde{F}$ be its regularized map so that the induced
map $\tilde{F}^{*}:\mathcal{P}ic(X) \to \mathcal{P}ic(X)$ satisfies
$(\tilde{F}^n)^{*} = (\tilde{F}^{*})^n.$ Then

\begin{theorem}\label{theo-diller}
    (See \cite{DF}) Let $F:\pr \to \pr$ be a birational map, $\tilde{F}$ be its regularized map and let $d_n = deg(F^n).$  Then up to bimeromorphic conjugacy, exactly one of
    the following holds:
    \begin{itemize}
        \item The sequence $d_n$ grows quadratically, $\tilde{F}$ is an automorphism and $f$ preserves an elliptic fibration.
        \item The sequence $d_n$ grows linearly and $f$ preserves a rational fibration. In this case $\tilde{F}$ cannot be conjugated to an automorphism.
        \item The sequence $d_n$ is bounded, $\tilde{F}$ is an automorphism and $f$ preserves two generically transverse rational fibrations.
        \item The sequence $d_n$ grows exponentially.
    \end{itemize}
    In the first three cases $\delta(F) = 1$ while in the last one $\delta(F) > 1.$ Furthermore in the first and second, the invariant fibrations are unique.
\end{theorem}

\subsection{Lists of orbits.}
We derive our results in the non-degenerate case by using Theorem
\ref{th_BK} below, established and proved in \cite{BK1}. The proof of
that is based in the same tools explained in the above paragraph. In
order to determine the matrix of the extended map in the Picard
group, it is necessary to distinguish between different behaviors
of the iterates of the map on the indeterminacy points of its
inverse.

The theorem is written for a general family $G$ of quadratic maps of
the form $G = L \circ J.$ As we will see the maps of  family
(\ref{eq1}), when the triangle is non-degenerate, are linearly
conjugated to such a maps. Here $L$ is an invertible linear map and
$J$ is the involution in $\pr$ as follows:
\begin{equation*}
J[x_0:x_1:x_2] = [{x_1}{x_2}:{x_0}{x_2}:{x_0}{x_1}].
\end{equation*}
We find that the involution $J$ has an indeterminacy locus
$\mathcal{I} = \{\epsilon_0, \epsilon_1, \epsilon_2\}$ and a set of
exceptional curves  $\mathcal{E} = \{\Sigma_0, \Sigma_1,
\Sigma_2\}$, where $\Sigma_i = \{x_i = 0\}$ for $i = 0,1,2,$ and
$\epsilon_i = \Sigma_j \bigcap \Sigma_k$ with $\{i,j,k\} =\{0,1,2\}$
and $i\ne j\ne k\,,\,i\ne k.$ Let $\mathcal{I}(G^{-1}) := \{a_0,
a_1, a_2\},$ the elements of this set are determined by $a_i:=
G(\Sigma_i - \mathcal{I}(J))=L \,\epsilon_i$ for $i = 0,1,2;$ see
\cite{BK1}.

To follow the orbits of the points of $\mathcal{I}(G^{-1})$ we need
to understand the following definitions and construction of lists of
orbits in order to apply the result of Theorem \ref{th_BK}.

We assemble the orbit of a point $p \in \pr$ under the map $G$ as
follows. For a point $p \in \mathcal{E}(G) \cup \mathcal{I}(G)$  we
say that the orbit $\mathcal{O}(p) = \{p\}$. Now consider that there
exits a $p \in \pr$ such that its $n^{th}-$ iterate belongs to
$\mathcal{E}(G) \cup \mathcal{I}(G)$ for some $n$, whereas all the
other $n-1$ iterates of $p$ under $G$ are never in $\mathcal{E}(G)
\cup \mathcal{I}(G)$. This is to say that for some $n$ the orbit of
$p$ reaches an exceptional curve of $G$ or an indeterminacy point of
$G.$ We thus define the orbit of $p$ as $\mathcal{O}(p) = \{p, G(p),
..., G^{n}(p)\}$ and we call it a  \textit{singular orbit}. If for
some $p \in \pr$ in turns out that $p$ and all of its iterates under
$G$ are never in $\mathcal{E}(G) \cup \mathcal{I}(G)$ for all $n,$
we set as $\mathcal{O}(p) = \{p, G(p), G^{2}(p) ...\}$ and
$\mathcal{O}(p)$ is \textit{non singular orbit}. We now make another
characterization of these orbits. Consider that a singular orbit
reaches an indeterminacy point of $G$, this is to say that $G^{n}(p)
\in \mathcal{I}(G)$ but its not in $\mathcal{E}(G).$  We call such
orbits as \textit{singular elementary orbits} and we  refer them as
SE-orbits. To apply Theorem \ref{th_BK} we need to organize our SE
orbits into lists in the following way.

Two orbits $\mathcal{O}_{1} = \{a_1,...,\epsilon_{j_1}\}$ and
$\mathcal{O}_{2} = \{a_2,...,\epsilon_{j_2}\}$ are in the same list
if either $j_1=2$ or $j_2=1,$ that is, if the ending index of one
orbit is the same as the beginning index of the other. We have the
following possibilities:

\begin{itemize}
\item \textit{Case 1: One SE-orbit,} $\mathcal{O}_{i} =
\{a_i,...,\epsilon_{\tau(i)}\}.$ Then we have the list
$\mathcal{L}=\{\mathcal{O}_{i} = \{a_i,...,\epsilon_{\tau(i)}\}\}.$
If $\tau(i)=i$ we say that $\mathcal{L}$ is a closed list. Otherwise
it is an open list.

\item \textit{Case 2: Two SE-orbits,} $\mathcal{O}_{i} =
\{a_i,...,\epsilon_{\tau(i)}\}\}$ and $\mathcal{O}_{j} =
\{a_i,...,\epsilon_{\tau(j)}\}\}.$
 In this case we can have either two closed lists,
\begin{center}$\mathcal{L}_{1} = \{\mathcal{O}_{i} = \{a_i,...,\epsilon_i\}\}$
\quad and \quad $\mathcal{L}_{2} = \{\mathcal{O}_{j} =
\{a_j,...,\epsilon_j\}\} \quad \text{with}\quad i\ne j$
\end{center}
or one open and one closed list

\begin{center}
$\mathcal{L}_1 = \{\mathcal{O}_{i} = \{a_i,...,\epsilon_i\}\}$ \quad
and \quad $\mathcal{L}_{2} = \{\mathcal{O}_{j} =
\{a_j,...,\epsilon_k\}\} \quad\text{with}\quad i\ne j\,,\,j\ne
k\,,\,k\ne i$
\end{center}
or a single list

\begin{center}
$\mathcal{L} = \{\mathcal{O}_{i} = \{a_i,...,\epsilon_j\},
\mathcal{O}_{j} =
\{a_j,...,\epsilon_{\tau(j)}\}\}\quad\text{with}\quad i\ne j$
\end{center}
which is closed if $\tau(j)=i$ and an open list otherwise.

Notice that we cannot have two open lists because there are at most
three SE-orbits.

\item \textit{Case 3: Three SE orbits}: In this case we can have either \textit{three closed lists}
\begin{center}
$\mathcal{L}_{1} = \{\mathcal{O}_{0} = \{a_0,...,\epsilon_0\}\}$
\quad and \quad $\mathcal{L}_{2} = \{\mathcal{O}_{1} =
\{a_1,...,\epsilon_1\}\}$ \quad and \quad $\mathcal{L}_{3} =
\{\mathcal{O}_{2} = \{a_2,...,\epsilon_2\}\},$
\end{center}
or \textit{two closed lists}
\begin{center}
$\mathcal{L}_{1} = \{\mathcal{O}_{i} = \{a_i,...,\epsilon_j\},
\mathcal{O}_{j} = \{a_j,...,\epsilon_i\}\}$  and $\mathcal{L}_{2} =
\{\mathcal{O}_{k} = \{a_k,...,\epsilon_k\}\}\quad \text{with} \quad
i\ne k \ne j\quad\text{and}\quad i\ne j$
\end{center}
or \textit{one closed list}
\begin{center}
$\mathcal{L} = \{\mathcal{O}_{0} = \{a_0,...,\epsilon_1\},
\mathcal{O}_{1} = \{a_1,...,\epsilon_2\}, \mathcal{O}_{2} =
\{a_2,...,\epsilon_0\}\}.$
\end{center}

\end{itemize}

We now define two polynomials $\mathcal{T}_{\mathcal{L}}$ and
$\mathcal{S}_{\mathcal{L}}$ which we will use to state theorem
\ref{th_BK}. Let $n_i$ denote the sum of the number of elements of
an orbit $\mathcal{O}_{i}$ and let $\mathcal{N}_{\mathcal{L}} = n_u
+ ...+ n_{u+\mu}$ denote the sum of the numbers of elements of each
list $\left|\mathcal{L}\right|.$ If $\mathcal{L}$ is closed then
$\mathcal{T}_{\mathcal{L}} = x^{\mathcal{N}_{\mathcal{L}}}-1$ and if
$\mathcal{L}$ is open then $\mathcal{T}_{\mathcal{L}} =
x^{\mathcal{N}_{\mathcal{L}}}.$ Now we define
$\mathcal{S}_{\mathcal{L}}$ for different lists as follows:
\begin{equation*}
{\mathcal{S}_{\mathcal{L}}}(x) = \left\{ {\begin{array}{*{20}{c}}

{1} & {} & \mbox{if}\;{\left| \mathcal{L} \right| = \left\{ {{n_1}} \right\},}  \\
   {{x^{{n_1}}} + {x^{{n_2}}} + 2} & {} & \mbox{if}\; \mathcal{L}\; \mbox{is closed and}\; {\left| \mathcal{L} \right| = \left\{ {{n_1},{n_2}} \right\},}  \\
   {{x^{{n_1}}} + {x^{{n_2}}} + 1} & {} & \mbox{if}\;\mathcal{L}\; \mbox{is open and}\; {\left| \mathcal{L} \right| = \left\{ {{n_1},{n_2}} \right\},}  \\
   {\sum\limits_{i = 1}^3 {\left[ {{x^{{\mathcal{N}_{\mathcal{L}}} - {n_i}}} + {x^{{n_i}}}} \right]  + 3}} & {} & \mbox{if}\; \mathcal{L}\; \mbox{is closed and}\; {\left| \mathcal{L} \right| = \left\{ {{n_1},{n_2},{n_3}} \right\},}  \\
   {\sum\limits_{i = 1}^3 {{x^{{\mathcal{N}_{\mathcal{L}}} - {n_i}}}}  + \sum\limits_{i \ne 2} {{x^{{n_i}}}}  + 1} & {} & \mbox{if}\;\mathcal{L}\; \mbox{is open and}\;{\left|\mathcal{L} \right| = \left\{ {{n_1},{n_2},{n_3}} \right\}.}  \\
 \end{array} } \right.
\end{equation*}
\begin{theorem}\label{th_BK}
(\cite{BK2})\,\,If $G = L \circ J$, then the dynamical degree
$\delta(G)$ is the largest real zero of the polynomial
$$\mathcal{X}(x) = (x - 2)\prod\limits_{\mathcal{L} \in {\mathcal{L}^c} \cup {\mathcal{L}^o}} {{\mathcal{T}_{\mathcal{L}}}(x) + (x - 1)\sum\limits_{\mathcal{L} \in {\mathcal{L}^c} \cup {\mathcal{L}^o}} {{S_L}(x)\prod\limits_{\mathcal{L}' \ne \mathcal{L}} {{\mathcal{T}_{\mathcal{L}'}}(x).} } }$$
Here $\mathcal{L}$ runs over all the orbit lists.
\end{theorem}

\section{Proof of the results}

$F$ is a non-degenerate when the sets $\mathcal{E}(F),\,\,\mathcal{E}(F^{-1})$ have the three
elements, each two of them intersecting on distinct points of $\mathcal{I}(F),\,\,\mathcal{I}(F^{-1})$ presented in section $2.1$. In this section we consider this to do the following study.

We consider the involution $J[x_0:x_1:x_2]$ introduced in section $2.3$,
and two invertible linear maps $M_1$ and $M_2$ of $\pr$
such that $M_1$ sends each $\Sigma_i$ to $S_i$ and
$M_2(A_i)=\epsilon_i$ for $i = 0,1,2.$ Then the mapping $M_2\circ
F\circ M_1$ is quadratic and sends each $\Sigma_i$ to $\epsilon_i.$
Therefore $M_2\circ F\circ M_1$ must be of the form
$[\lambda_0\,x_1\,x_2:\lambda_1x_0x_2:\lambda_2\,x_0\,x_1]$ with
$\lambda_i\ne 0$ for each $i=0,1,2,$ that is $M_2\circ F\circ
M_1=D\circ J$ where $D$=diag$(\lambda_0,\lambda_1,\lambda_2).$
Calling $L=M_1^{-1}\circ M_2^{-1}\circ D$ we get that $F\circ M_1=
M_2^{-1}\circ D\circ J=M_1\circ L\circ J,$ that is the mappings $F$
and $G:=L\circ J$ are linearly conjugated. Calling
$a_i:=G(\Sigma_i-\mathcal{I}(J))=L \epsilon_i$ for $i=0,1,2$ we are
going to identify each $a_i\in \mathcal{I}(G^{-1})$ with $A_i\in
\mathcal{I}(F^{-1}).$

From now on we are going to assume that $f(x,y)$ is birational (see
conditions $(a)$ in Lemma \ref{conditions}), that $F[x_0:x_1:x_2]$
has degree two and that it is not not degenerated (i. e.,
$(\alpha\gamma)_{12}\ne 0 \ne (\beta\gamma)_{12}).$ The exceptional
set of $F$ and $F^{-1}$ can be seen in the following figure:


The mapping $F$ is bijective from $\pr\setminus \{S_0,S_1,S_2\}$ to
$\pr\setminus \{T_0,T_1,T_2\}.$ The only points in $T_i$ which have
preimage by $F$ are $A_j,A_k$ with $i\notin\{j,k\}\,,\,j\ne k$ which
have as preimages $S_j$ and $S_k$ respectively.

To prove the results we use the following strategy. First we perform
the necessary blow-up's in order to have an extension of $F,$ that
is $\tilde F$ and it is AS. Then we construct the lists of the
orbits of points $A_i$ and we apply Theorem \ref{th_BK}. We now give
the proofs of Theorems $1$ to $4$. They are as follows.
\begin{enumerate}
\item \textbf{Proof of Theorem $\bold{1}$}
\begin{proof}

The conditions on the parameters imply that $F(A_0) = A_0$ with
$A_0\notin\mathcal{I}(F)$ and $F(A_1) = A_0.$ Since $A_0, A_1 \in
S_0\in \mathcal{E}(F)$, thus we find that their orbits are
$\mathcal{O}_0 = \{A_0\}$ and $\mathcal{O}_1 = \{A_1\}$ which are
singular but not elementary. Now it remains to analyze the behavior
of iterates of $A_2$. We claim that $\nexists p\in\N: {F^p}({A_2}) =
{O_1}$ and $\nexists p\in\N: {F^p}({A_2}) = {O_2}.$ It is so because
if ${F^p}({A_2}) = {O_1}$ then, since $O_1\in S_0=T_2$ it would
imply $O_1=A_0$ or $O_1=A_1,$ that is $\alpha_1=0$ or $\alpha_2=0.$
 Similarly, ${F^p}({A_2}) = {O_2}$ implies $\gamma_1=0$ or
$\gamma_2=0.$ Therefore the only possibility is that any iterate of
$A_2$ reaches $O_0.$ Thus we assume the following cases:
\begin{enumerate}

\item Assume that ${F^p}({A_2}) \ne {O_0}$ for all $p \in \N.$ Then the map $F$ is itself AS. Hence, $\delta(f) = 2.$

\item Now assume that ${F^p}({A_2}) = {O_0}$ for some $p \in \N.$ Thus we have a SE orbit of $A_2$ which is:
$\mathcal{O}_{2} = \{A_2, {F}({A_2}),..., {F^p}({A_2})= O_0\}.$
In this case we have only
one list $\mathcal{L}_{o}$ which is open. That is:
$$\mathcal{L}_{o} = \{\mathcal{O}_{2} = \{A_2, {F}({A_2}),..., {F^p}({A_2})= O_0\}\}.$$
To find the characteristic polynomial we use Theorem \ref{th_BK}. We
find that $\mathcal{N}_{\mathcal{L}_o} = p $,
$\mathcal{T}_{\mathcal{L}_o} = x^{p}$ and
$\mathcal{S}_{\mathcal{L}_o} = 1.$ Then the $\delta(F)$ is the
largest root of the characteristic polynomial
$\mathcal{Y}_p(x):={x^{p + 2}} - 2{x^{p + 1}} + {x} - 1.$

Observe that for all the values of $p \in \N$ the above polynomial
has always the largest root $\lambda > 1.$ This is because
$\mathcal{Y}_p(1) = -1 < 0$ and $\mathcal{Y}_p(2) = 1 > 0,$
therefore there always exists a root $\lambda > 1$ such that
$\mathcal{Y}_p(\lambda) = 0.$ Hence $d_n$ has exponential growth
rate.
\end{enumerate}
\end{proof}

\item \textbf{Proof of Theorem $\bold{2}$}

Assume that $\gamma_1 = 0.$ From Lemma \ref{conditions} we know that
$\alpha_1, \beta_1$ and $\gamma_2$ are non zero. We distinguish two
cases, depending on $\alpha_2.$
\begin{itemize}

\item Consider the case when $\alpha_2=0.$
Observe that $S_0 \twoheadrightarrow A_0 = O_2$ and
$S_1\twoheadrightarrow A_1 = O_1.$ Hence we blow up the points
$A_0,\,A_1$ to get the exceptional fibres $E_0,\,E_1.$ Let $X$ be
the new space and let $\tilde{F}: X \to X$ be the extended map on
$X$. In order to know $\tilde{F}$ we see $[u:v]_{E_0}\in S_0$ (resp.
$[u:v]_{E_1}\in S_0$) as $\lim_{t\to 0}[tu:1:tv]$ (resp. $\lim_{t\to
0}[tu:tv:1]$), we evaluate $F[tu:1:tv]$ (resp. $F[tu:tv:1]$) and
take limits again. We get:
$$\tilde{F}[0:x_1:x_2]=[x_2:x_1+\beta
x_2]_{E_0}\,,\,\tilde{F}[u:v]_{E_0}=[0:\alpha_1 v:u]\in T_2=S_0$$
and
$$\begin{array}{lr} \tilde{F}[x_0:x_1:-\frac{\gamma_0}{\gamma_2}x_0]=[x_0:\alpha_0 x_0+\alpha_1
x_1]_{E_1}\\ \tilde{F}[u:v]_{E_1}=[\gamma_2 u:\gamma_2(\alpha_0
u+\alpha_1 v):\beta_2 u]\in T_1\end{array}$$

 Then the map $\tilde{F}$ sends the curve $S_0\to E_0 \to
S_0$ and $S_1 \to E_1\to T_1.$ We observe that no new point of
indeterminacy is created therefore $\mathcal{{I}}(\tilde{F}) =
\left\{ {{O_0}} \right\}$ and $\mathcal{{E}}(\tilde{F}) = \left\{
{{S_2}} \right\}.$ Assume that there exists $p\in\N$ such that
$\tilde{F}^p(A_2)=O_0.$ Then we blow up
$A_2,\tilde{F}(A_2),\tilde{F^2}(A_2),\ldots,\tilde{F^p}(A_2)=O_0$
getting the exceptional fibres which we call
$E_2,E_3,\ldots,E_{p+2}.$ Set $\tilde{F_1}:X_1 \to X_1$ the extended
map. Performing the blow up at $O_0,$ since $T_0$ is sent to $O_0$
via $F^{-1},$ we have that $\tilde{F}_1^{-1}:T_0\to E_{p+2}.$ Then
 $S_2 \to E_2 \to E_3 \to \cdots\to E_{p+1}\to E_{p+2}
\to T_0.$ Hence $\tilde{F}_1:X_1\to X_1$ is an AS map and also an automorphism. Now we have two closed lists as follows
$$\mathcal{L}_{c_1} = \{\mathcal{O}_0 = \{A_0 = O_2\},\quad \mathcal{O}_2 = \{A_2,\,\tilde{F}(A_2)\,,\ldots,\tilde{F}^{p}(A_2)=O_0\}\},$$
$$\mathcal{L}_{c_2} = \{\mathcal{O}_1 = \{A_1 = O_1\}\}.$$
Then by using Theorem \ref{th_BK} we find that the characteristic
polynomial associated to $F$ is $\mathcal{X} =
(x^{p+1}+1)(x-1)^2(x+1).$ If $p$ is even then $x^{p+1}+1$ has the
factor $x+1$ and $\mathcal{X} =
(x-1)^2\,(x+1)^2\,(x^p-x^{p-1}+\cdots -x+1).$ Hence the sequence of
degrees is $d_{n} = c_0+c_1\, n+
c_2\,(-1)^n+c_3\,n\,(-1)^n+c_4\,\lambda_1^n + c_5\, \lambda_2^n
+...+c_{p+3}\,\lambda_{p}^n,$ where $c_i$ are constants and
$\lambda_1,\,\lambda_2,...,\lambda_{p}$ are the roots of polynomial
$x^p-x^{p-1} +\cdots -x+1.$ By looking at $d_n$ we see that $f$ does not grow quadratically or exponentially. As our map
$\tilde{F}_1$ is an automorphism then by using the results from Diller
and Favre in \cite{DF} we see that also cannot have linear growth. Therefore we must
have $c_1=c_3=0.$ Hence the sequence of degrees must be periodic.
This implies that $d_{2p+2+n} = d_{n}$ i.e. the sequence of degrees
is periodic with period $2p+2.$ If $p$ is odd then $d_n$ is also periodic of period $2p+2.$

If $\tilde{F}^{p}(A_2)\ne O_0$ for all $p\in\N,$ then we have two
lists which are open and closed as follows:
$$\mathcal{L}_{o} = \{ \mathcal{O}_{0} = \{A_0 = O_2\}\}\quad,\quad \mathcal{L}_{c} = \{ \mathcal{O}_{1} = \{A_1 = O_1\}\}.$$ Then $\delta(F)$ is determined by the polynomial $(x-1)^2(x+1),$ and $\delta(f)=1.$
The sequence of degrees is
$d_n=\frac{5}{4}+\frac{1}{2}\,n-\frac{1}{4}\,(-1)^n.$


\item Now consider that $\alpha_2 \ne 0.$ The
parameters $\alpha_1, \beta_1, \gamma_2$ are all non zero. Observe
that $S_0 \twoheadrightarrow A_0 = O_2.$ The orbit of $A_0$ is SE.
By blowing up $A_0$ we get the exceptional fibre $E_0$ and the new
space $X.$ The induced map $\tilde{F}: X \to X$ sends the curve $S_0
\to E_0 \to S_0.$ Observe that now
$\mathcal{{I}}(\tilde{F}) = \left\{ {{O_0},{O_1}} \right\}$ and
$\mathcal{{E}}(\tilde{F}) = \left\{{{S_1},{S_2}} \right\}.$

We see that $A_1 \ne O_1$ and the exceptional curve $S_1 \twoheadrightarrow A_1 \in S_0.$ We observe
that the collision of orbits discussed in preliminaries is happening here. The orbit of $A_1$ under $\tilde{F}$ is as follows:
\begin{equation*}
S_1 \twoheadrightarrow  A_1 \to [\gamma_2 :\beta_2]_{E_0} \to [0 :
\alpha_1(\gamma_0+\beta_2): \beta_1]\in S_0 \to \cdots
\end{equation*}
After some iterates we can write the expression of
$\tilde{F}^{2k}(A_1)$ for all $k > 0 \in \N$ as
$\tilde{F}^{2k}(A_1) = [0 :
\alpha_1(\gamma_0+\beta_2)(1+\alpha_1+\alpha_1^2+\cdots+\alpha_1^{k-1}):
\beta_1]\in S_0.$
Observe that for some value of $k \in \N$ it is possible that
$\tilde{F}^{2k}(A_1) = O_1.$ This happens when the following
\textit{condition $k$}
 is satisfied for some $k.$
\begin{equation}\label{eq4}
 \alpha_1^2(\gamma_0+\beta_2)(1+\alpha_1+\alpha_1^2+\cdots+\alpha_1^{k-1})+ \alpha_2\beta_1=0.
\end{equation}
For such $k \in \N$ the orbit of $A_1$ is SE. By blowing up the points of this orbit we get the new space $X_1$ and the induced map $\tilde{F}_1.$ Then under the action of $\tilde{F}_1$ we have
\begin{equation*}\label{seq2} S_1 \to G_0 \to G_1 \to G_2 \to \cdots
\to G_{2k-1}\to G_{2k}\to T_1.\end{equation*}
Then $\mathcal{I}(\tilde{F}_1)=\{O_0\}$ and $\mathcal{E}(\tilde{F}_1)=\{S_2\}.$

Now if the orbit of $A_1$ is SE and if $\tilde{F}_1^{p}(A_2) = O_0$
that is the orbit of $A_2$ is also SE for some $p \in \N$ then we
have three SE orbits. If \textit{condition $k$} is not satisfied
then with the extended map $\tilde{F}$ we have
$\mathcal{I}(\tilde{F})=\{O_0,O_1\}.$ Therefore we have two options:
$\tilde{F}^{p}(A_2) = O_0$ or $\tilde{F}^{p}(A_2) = O_1.$

We claim that for all $p\in\N,$ $\tilde{F}^{p}(A_2) \ne  O_1.$
Assume that $\tilde{F}^{p}(A_2)= O_1$ and assume that
$F^j(A_2)\notin S_0$ for $j=1,2,\ldots, p-1.$
$\tilde{F}^{p}(A_2)={F}^{p}(A_2)=O_1.$ Since $O_1\in S_0$ and
$A_2\notin S_0$ if $F^p(A_2)= O_1$ then $p$ would be greater than
zero and since $S_0=T_2,$ it would imply that $O_1=A_1$ or
$O_1=A_2,$ which is not the case (recall that the only points in
$T_2$ which have a preimage are $A_1$ and $A_2$).

Contrarily, if it exists some $l\in\N,l<p$ such that $F^j(A_2)\notin
S_0$ for $j=1,2,\ldots, l-1$ but $F^l(A_2)\in S_0\setminus\{O_1\}$
then $F^l(A_2)$ must be equal to $A_1$ or $A_2$ that is,
${F}^{l}(A_2)= A_1$ or ${F}^{l}(A_2)= A_2.$ The second case is not
possible as $A_2$ is a fixed point. In the first case
$\tilde{F}^{p}(A_2)=\tilde{F}^{p-l}(F^l(A_2))=\tilde{F}^{p-l}(A_1)=O_1$
which implies that $p=l+2r$ and $\tilde{F}^{2r}(A_1)=O_1.$  Hence
the orbit of $A_1$ must be SE and that condition $k$ must be
satisfied for $k=r$  which is a contradiction. It implies that the
only available possibility for $\mathcal{O}_2$ to be SE is to have
that for some $p,$ $\tilde{F}^{p}(A_2) = O_0.$ After the blow up
process we get
$$S_2 \to E_1 \to E_2 \to \cdots\to E_{p}\to E_{p+1} \to T_0.$$
The extended map $\tilde{F}_2$ is an automorphism when we have three SE orbits.

The above discussion gives us three different cases.
\begin{itemize}
    \item One $SE$ orbit: This happens when $A_0 = O_2$ with the conditions that $\tilde{F}^{2k}(A_1)\ne O_1$ and
    $\tilde{F}^p(A_2)\ne O_0$ for all $k,p\in\N.$
    Therefore we have only one list $\mathcal{L}_{o}$ which is open that is
    $\mathcal{L}_{o} = \{ \mathcal{O}_{0} = \{A_0 = O_2\}\}. $ By using theorem \ref{th_BK} we find that $\delta(F)=\delta^{*} = \frac{\sqrt{5}+1}{2}.$ which is given by the greatest root of the
    polynomial $X(x) = {x^2} -x- 1.$ Therefore it has exponential growth.
\item Two $SE$ orbits $(a)$: It is the case when $A_0 = O_2, \,\,\tilde{F}^{p}(A_2) = O_0$ and $\tilde{F}^{2k}(A_1) \ne O_1$
for all $k\in\N.$
By organizing the orbits into lists we have one closed list
$\mathcal{L}_c = \{\mathcal{O}_0 = \{A_0 = O_2\},\quad \mathcal{O}_2 = \{A_2,\,\tilde{F}(A_2)\,,...,\tilde{F}^{p}(A_2)=O_0\}\}.$
By utilizing theorem \ref{th_BK} we find that the characteristic polynomial
associated to $F$ is $\mathcal{X}_p = x^{p+1}(x^2-x-1)+x^2.$
For $p=0$ and $p=1$ the sequence of degrees satisfies
$d_{n+3} = d_{n} $ and  $d_{n+4} = d_{n+3}$ respectively which corresponds towards boundedness of $f$.

For $p = 2$ we get the polynomial $\mathcal{X}_2 = x^2(x+1)(x-1)^2.$
Looking at the first degrees we get that the sequence of degrees is
$d_n = -1+2\,n.$

For $p > 2,$ we observe that  $\mathcal{X}_p(1)=0,\,\mathcal{X'}_p(1)=2-p<0$ and $\lim_{x\to +\infty}
\mathcal{X}_p(x)=+\infty.$ Hence  $\mathcal{X}_p$ always has a root
$\lambda>1$ and the result follows.

\item Two $SE$ orbits $(b)$: When we have $A_0 = O_2, \,\,\tilde{F_1}^{2k}(A_1)
=O_1$ and $\tilde{F_1}^{p}(A_2)\ne O_0$ for all $p\in\N$
then there is one open and one closed list and
$\mathcal{X}_k = x^{2k+1}(x^2-x-1)+1.$
We observe that for all the values of $k \in \N\,,\,k\ge 1$ the
polynomial $\mathcal{X}_k$ has always a root
$\lambda > 1.$ Therefore $f$ has exponential growth.

\item Three $SE$ orbits: In this case we have $A_0 = O_2,\,\,\tilde{F}^{2k}(A_1) = O_1,\,\,\tilde{F}^{p}(A_2) = O_0,$ for
a certain $p,k\in\N.$
We have two closed lists as follows:
$$\mathcal{L}_c = \{\mathcal{O}_0 = \{A_0 = O_2\},\quad \mathcal{O}_2 = \{A_2,\,\tilde{F}(A_2)\,,...,\tilde{F}^{p}(A_2)=O_0\}\},$$
$$\mathcal{L}_c= \{\mathcal{O}_1 = \{A_1,\,\tilde{F}(A_1)_{E_0}\,,...,\tilde{F}^{2k}(A_1)_{S_0}=O_1\}\}.$$
From theorem \ref{th_BK} we can write $\mathcal{X}_{(k,p)} = x^{p+1}(x^{2k+3}-x^{2k+2}-x^{2k+1}+1)+x^{2k+3}-x^2-x+1.$
The map $\tilde{F}_2$ is an automorphism for
all the values $(k,p).$ According to Diller and Favre in \cite{DF} the
growth of degrees of iterates of an automorphism could be bounded,
quadratic or exponential but it cannot be linear as in such a case
the map is never an automorphism. For this we observe the behavior of $\mathcal{X}_{(k,p)}$ around $x=1.$ we consider it's Taylor expansion near $x = 1:$
$$ \mathcal{X}_{(k,p)}(x) = 2(2-kp+2k)(x-1)^2 +
O(\left|x-1\right|^3).$$ Thus $\mathcal{X}_{(k,p)}$ vanishes at $x=1$
 $x = 1$ and has a maximum on it $p>\frac{2(1+k)}{k}.$ Since $\lim_{x\to +\infty
\mathcal{X}_{(k,p)}(x)}=+\infty,$ always exists a root greater than
one. If $p\le \frac{2(1+k)}{k}\,,\,k\ge 1$ then
the pairs $(k,p)$ are in the
set: $A_{(k,p)} = \{((k \geq 1),0),\,((k \geq 1),1),\,((k \geq
1),2),\,(1,3),\,(2,3),\,(1,4)\}.$

For $(k,p) = (k,0),$ when $k$ is even the
sequence of degrees is $$d_{n} = c_0+c_1 n+c_2\,(-1)^n +
c_3\,(-1)^n\,n + c_4\,\lambda_1^n + c_5\, \lambda_2^n
+...+c_{2k+3}\,\lambda_{2k}^n,$$ where $c_i$ are constants and
$\lambda$'s are the roots of polynomial
$x^{2k+2}=1$ different from $\pm 1.$  If $k$ is odd then $$d_{n} =
l_0+l_1 n+l_2\,(-1)^n + l_3\,\mu_1^n+ l_4\,\mu_2^n +
\cdots+l_{2k+3}\,\mu_{2k+1}^n,$$ where $l_i$ are constants and
$\mu$'s are the roots of polynomial
$(x^{k+1}-1)\,(x^k+x^{k-1}+\cdots +x+1).$
Since $\tilde{F}_2$ is an automorphism for all ${(k,p)},$ using \cite{DF}
we have $c_1 = 0 =  c_3$ and also $l_1=0.$
This implies that $d_{2k+2+n} = d_{n},$ i. e., the sequence of degrees is periodic with period $2k+2.$
The argument for the proof of other values of $(k,p) \in A_{(k,p)}$ follows accordingly.

\end{itemize}
\end{itemize}


\item \textbf{Proof of Theorem $\bold{3}$}

From hypothesis and from Lemma \ref{conditions} we know that
$\alpha_2 \gamma_1 \ne 0$ and $\beta_2 \gamma_1 \ne 0$ therefore
$\alpha_2, \beta_2$ and $\gamma_1$ cannot be zero. There exist two
different cases to study depending on $\alpha_1.$

\begin{itemize}
\item
Consider that $\alpha_1 = 0.$ Then $A_0 =
O_1$ and $A_1 = O_2.$ We get new
space $X$ by blowing up $A_0$ and $A_1.$ $E_0,\,E_1$ are the exceptional fibres on these points
respectively. The extended map $\tilde{F}:X \to X$ sends $S_0 \to E_0
\to T_1$ and $S_1 \to E_1 \to T_2.$ Therefore the orbits of $A_0$ and $A_1$ are
SE. No new indeterminacy points have appeared therefore
$\mathcal{I}(\tilde{F}) = \left\{ {{O_0}} \right\}$ and
$\mathcal{E}(\tilde{F}) = \left\{ {{S_2}} \right\}.$ If
$\tilde{F}^p(A_2)=O_0$ for some $p\in\N$ then the orbit of $A_2$ is
SE. Let $\tilde{F}_2:X_1 \to X_1$ be the extended map on new space $X_1$ we get after blowing up the
points of orbit of $A_2.$ Then $\tilde{F}_2 $ sends
$S_2 \to E_2 \to E_3 \to \cdots\to E_{p+2}\to T_0.$ Then $\tilde{F}_2$ is an AS map and is an automorphism.

We see that we have one closed list and by utilizing Theorem \ref{th_BK} the characteristic polynomial
associated to $F$ is $\mathcal{Y}_p = x^{p+1}(x^3-x-1)+(x^3+x^2-1).$
For $p =0,$ the sequence of degrees $d_n = c_1 + c_2\,(-1)^n + c_3 (\lambda_1)^n+ c_4
(\lambda_2)^n,$ where $\lambda_1,\,\lambda_2$ are the two roots of $x^{2}+x+1=0.$ Hence 
$d_n$ satisfies $d_{n+6} = d_n,$ i.e., it is periodic of period $6.$
For $p\leq5$ the argument for the proof is similar with periods
$5,\,8,\,12,\,18$ and $30$ accordingly. When $p=6,$ the sequence of
degrees $d_n = c_1 + c_2\,n +c_3\,n^2 +c_4\,(-1)^n
+c_5\,(\lambda_1)^n + c_6 (\lambda_2)^n+ c_7 (\lambda_3)^n+ c_8
(\lambda_4)^n+ c_9 (\lambda_5)^n+ c_{10} (\lambda_6)^n.$ As
$\tilde{F}_2$ is an automorphism, from Theorem \ref{theo-diller},
the sequence of degrees
 does not grow linearly. Then either, $c_3 \ne 0$ and $d_n$ grows quadratically or $c_2 = 0 = c_3$ and $d_n$ is periodic of period $30.$

For $p > 6\,$ there always exists a
root $\lambda > 1$ of $\mathcal{Y}_p.$ Hence the
sequence of degrees grows exponentially for such values of $p$.

Now suppose that no $p$ exists such that $\tilde{F}^{p}(A_2) = O_0.$ In this case $\delta(F)$ is given
by the greatest real root of the polynomial $\mathcal{Y}(x)= x^3-x-1$.

\item Now consider that $\alpha_1 \ne 0.$ Observe that
in general $S_0\twoheadrightarrow A_0 \neq O_i$ for any $i \in \{0,1,2\}$ and
$F(A_0) = A_0.$ Thus $\mathcal{O}_0 = \{A_0\}$ is not a SE orbit. Now $S_1
\twoheadrightarrow A_1 = O_2.$ We blow up the point $A_1 = O_2.$ Therefore the orbit
of $A_1$ is SE. Let
$X$ be the new space after blowing up $A_1$ and let $E_1$ be the
exceptional fibre at this point. The induced map $\tilde{F}: X \to
X$ sends the curve $S_1 \to E_1 \to T_2 = S_0.$
Then
$\mathcal{I}(\tilde{F}) = \left\{ {{O_0\,,O_1}} \right\}$ and
$E(\tilde{F}) = \left\{ {{S_0\,,S_2}} \right\}.$

The curve $S_2 \twoheadrightarrow A_2$ and $O_1 \in S_0 = T_2.$ Note
that $F^p(A_2) \neq O_1.$ As the only points on $T_2$ which have
preimages are $A_0$ and $A_1.$ Then if the orbit of $A_2$ reaches
$O_1$ at some iterate of $F$ then $O_1$ should be equal to either
$A_0$ or $A_1.$ As $\alpha_1 \neq 0$ hence $A_0\neq O_1 \neq A_1.$
This implies that $F^p(A_2) \neq O_1$ for all $p$ but it is possible
 that $\tilde{F}^p(A_2)= O_1.$ Then in general there are two possibilities:  $\tilde{F}^{p}(A_2) = O_0$ for some $p \in \N$ or $\tilde{F}^q(A_2)= O_1$
 for some $q \in \N.$ In both cases the orbit of $A_2$ is SE.

Now if there exists some $p\in\N$ such that $\tilde{F}^p(A_2)=O_0,$ then to get
$X_1$ we blow-up all the points of the orbit of $A_2.$ The extended map $\tilde{F}_1:X_1 \to X_1$ sends
$S_2 \to E_2 \to E_3 \to \cdots\to E_{p+2}\to T_0.$ This shows that $\tilde{F}_1$ is an AS map.

Now we have one open list and the characteristic polynomial
associated to $F$ is $\mathcal{Y}_p = x^{p+1}(x^2-x-1)+x^2-1.$
Observe that for all the values of $p \in \N$ the polynomial
$\mathcal{Y}_p $ always has the largest root $\lambda > 1.$
Hence $d_n$ grows exponentially and $\delta(F)$ approaches to the
value $\delta^{*} = \frac{1+\sqrt{5}}{2}$ as $p \to \infty.$

Also if there exists some $q\in\N$ such that $\tilde{F}^q(A_2)=O_1,$ then $\tilde{F}_1$ is an AS map.
We have one closed list and the characteristic polynomial associated to $F$ is
$\mathcal{Y}_q = x^{q+1}(x^2-x-1)+x^2.$
Note that there are no mappings for $q \in \{0\,,1\}.$ As $A_2[1] \neq O_1[1]$ therefore $q = 0$ is not possible.
For $q = 1$ we have two possibilities. First when $A_2 \notin S_1,$ then the condition $\tilde{F}(A_2) = {F}(A_2) = O_1.$ But the orbit of $A_2$ can never reach $O_1$ because $O_1 \in T_2$ and $O_1 \neq A_0.$
 Now if $A_2 \in S_1$
then we have the condition $\tilde{F}(A_2) = O_1.$ In this case $\tilde{F}(A_2) \in E_1$ but it is clear that $O_1 \notin E_1$ therefore $q=1$ is not possible.

For $q = 2$ we get the polynomial $x^2(x+1)(x-1)^2$ and the sequence
of degrees is $d_n = c_2\,(-1)^n + c_3+ c_4\,n.$ Looking at the
first degrees we get $d_n=-1+2n.$ For $q
> 2,$ we observe that $\mathcal{Y}_q$ always has a root $\lambda>1$
and the result follows.

\end{itemize}

\item \textbf{Proof of Theorem $\bold{4}$}

Considering the hypothesis we know that $\gamma_1 \neq 0 \neq
\gamma_2$ therefore from lemma \ref{conditions} the parameters
$\{\alpha_0\,,\beta_0,\,\beta_1,\,\beta_2,\,\gamma_0\}$ and one of
$\{\alpha_1\,,\alpha_2\}$ at the same moment can be zero. Hence two
different cases, $\alpha_1=0$ and $\alpha_2=0$ are considered as
follows:
\begin{itemize}
\item
Consider that $\alpha_1 = 0$ and $\alpha_2 \neq 0 .$

Observe that $S_0 \twoheadrightarrow A_0 = O_1.$ Let $X$ be the new
space we get after blowing up the point $A_0$ and let $E_0$ be the
exceptional fibre at this point. The induced map $\tilde{F}: X \to
X$ sends the curves $S_0\to E_0 \to T_1.$ Hence
$\mathcal{{I}}(\tilde{F}) = \left\{ {{O_0\,,O_2}} \right\}$ and
$\mathcal{{E}}(\tilde{F}) = \left\{ {{S_1\,,S_2}} \right\}.$ Note
that $A_1$ is not indeterminate for $F$ and $A_1 \in S_0$ therefore orbit of $A_1$ collides with $A_0.$ But we observe that $\tilde{F}^q(A_1) \neq O_2$ for all $q$ as $O_2 \in S_0 = T_2.$
However it is possible that $\tilde{F}^p(A_1)= O_0$ for
some $p \in \N.$ If there exists such $p$ then we blow up the points
of the orbit of $A_1$ to get the exceptional fibres $E_i$'s. Let
$X_1$ be the new space. Then the extended map $\tilde{F}_1$ sends
$S_1 \to E_1 \to E_2 \to \cdots\to E_{p+1}\to T_0.$
Now $\mathcal{{I}}(\tilde{F}_1) = \left\{ {{O_2}} \right\}$ and
$\mathcal{{E}}(\tilde{F}_1) = \left\{ {{S_2}} \right\}.$ The
exceptional curve $S_2 \twoheadrightarrow A_2.$ But the orbit of
$A_2$ can never reach $O_2$ as $O_2 \in T_2$. Thus the orbit of
$A_2$ is not SE. Hence the map $\tilde{F}_1$ is AS.

In this case we have one closed list by using Theorem \ref{th_BK} the characteristic polynomial
associated to $F$ is $\mathcal{Z}_p = x^{p+1}(x^2-x-1)+x^2.$ The
proof for all values of $p$ is similar to the last part of theorem
\ref{theo3}.

Now assume that $\tilde{F}^p(A_1)\ne O_0$ for all $p\in\N$ such that the orbit of $A_1$ is not SE then for
some $q$ it is possible that $\tilde{F}^{q}(A_2) = O_0.$ In this case after the blow up of the orbit
of $A_2,$ in the extended space $X_1$ the induced map $\tilde{F}_1$
acts as
$S_2 \to E_1 \to E_2 \to \cdots\to E_{q+1}\to T_0.$ Thus the orbit of $A_2$ is SE and the map $\tilde{F}_1$ is AS.

We have one open list and the characteristic polynomial
associated to $F$ is $\mathcal{Z}_q = x^{q+1}(x^2-x-1)+x^2-1.$
We observe that for all the values of $q \in \N$ the polynomial
$\mathcal{Z}_q$ always the largest root $\lambda > 1$ and $d_n$ grows exponentially.

We now consider the case when for some $p,\,q \in \N$ we have
$\tilde{F}^p(A_1)= O_0$ and $\tilde{F}^q(A_2)= O_0.$


We claim that $p$ must be different from $q.$ Assume that
$\tilde{F}^k(A_1) \ne A_2$ and $\tilde{F}^j(A_2) \ne A_1$ for any $0
< k < p$ and $0 < j < q.$ Because otherwise these points can have
multiple preimages. Then $p = q$ gives the condition that
$\tilde{F}^p(A_1)= \tilde{F}^p(A_2) = O_0$ implies that $A_1 = A_2,$ as $F$ is bijective except for some particular points. But this gives a contradiction as all $A_i$'s must be different in this case. Now if there exists some $k$ or
$j$ such that $\tilde{F}^k(A_1) = A_2$ or $\tilde{F}^j(A_2) = A_1$
then there is collision of orbits.
This gives that either, $k>0$ or $j>0$ which shows that $p \ne q.$

Now consider that $q > p.$ Then the orbit of $A_2$ must collides with the
orbit of $A_1.$  Because if not then this claims that
$\tilde{F}^k(A_2) \ne A_1$ for all $0 < k < q.$ As $q > p$ then for
some $j > 0$ we can write $q = p + j.$ This gives
$\tilde{F}^{q}(A_2) = \tilde{F}^{j+p}(A_2) = O_0 =
\tilde{F}^{p}(A_1).$ As $O_0$ has unique preimage and there is no
collision this implies that no orbit enters any $T_1$ or $T_2.$
Therefore the points $\tilde{F}^{j+p}(A_2)$ and $\tilde{F}^{p}(A_1)$
have unique preimages. Then for $\tilde{F}^{j+p}(A_2) = O_0 =
\tilde{F}^{p}(A_1)$ we can find the preimages by iterating $p$ times
with $\tilde{F}^{-1}.$ This gives us
$\tilde{F}^{-p}(\tilde{F}^{j+p}(A_2)) =
\tilde{F}^{-p}(\tilde{F}^{p}(A_1))$ which implies that
$\tilde{F}^j(A_2) = A_1$ for some $0 < j < q,$ which gives
contradiction to our claim. This implies that in the case when $q >
p$ or $q < p$ we always have collision of orbits of $A_2$ with $A_1$
or $A_1$ with $A_2.$

Now for $q > p$ we must have $\tilde{F}^k(A_2) = A_1$ for some $0 <
k < q.$ Then we see that:
$$S_2 \twoheadrightarrow A_2 \to \tilde{F}(A_2) \to \cdots \to \tilde{F}^{q-p}(A_2)= A_1\to  \tilde{F}(A_1)\to \cdots \to \tilde{F}^p(A_1)= O_0.$$
This implies that the orbit of $A_2$ is no more SE. Hence two SE orbits are the orbits of $A_0$ and $A_1.$
This shows that the
characteristic polynomial is $Z_p$ in this case. Similarly, in the
second case the characteristic polynomial is $Z_q.$

Finally, if $\tilde{F}^p(A_1)\ne O_0$ and $\tilde{F}^p(A_2)\ne O_0$
for any $p \in \N$ then we have one SE orbit and the characteristic polynomial
is given by $\mathcal{Z}(x)= x^2-x-1.$ Then the dynamical degree
$\delta(f)=\delta^{*}.$

\item
Now consider $\alpha_2 = 0,$ from Lemma \ref{conditions} we have $\alpha_1\,,\gamma_1 \,,\gamma_2$
non zero. Observe that $A_0$ is not an indeterminate for $F$ and is a fixed point of
$F.$ Hence the orbit of $A_0$ is not SE.

Now $S_1 \twoheadrightarrow A_1 = O_1.$ After blowing up the point $A_1$ to get the exceptional fibre $E_1,$
the induced map $\tilde{F}: X \to X$ for new space $X$ sends the curve $S_1\to E_1 \to T_1.$
Now $\mathcal{{I}}(\tilde{F}) = \left\{ {{O_0\,,O_2}} \right\}$ and
$\mathcal{{E}}(\tilde{F}) = \left\{ {{S_0\,,S_2}} \right\}.$ The
exceptional curve $S_2 \twoheadrightarrow A_2.$ Note that $\tilde{F}^p(A_2) \ne O_2$ for all $p\in \N$ as $O_2 \in T_2.$
Then for some $p \in \N$ it is possible that $\tilde{F}^{p}(A_2) = O_0.$ In
this case the orbit of $A_2$ is SE. Let ${X_1}$ be the expanded
space we get after blowing up the orbit of $A_2$ the observe that now $\tilde{F}_1 : {X_1} \to {X_1}$ is an AS map but is not an automorphism as $S_0$ still collapses.

In this case we have one closed list and one open and the characteristic polynomial
associated to $F$ is $\mathcal{Z}_p = (x^{p+1}+1)(x-1)^2.$

Now if no such $p$ exists so that $\tilde{F}^{p}(A_2) = O_0$ then we have one open list $\delta(F)$
is given by the largest root of the polynomial $\mathcal{Z}(x)=
(x-1)^2,$ which is one.

\end{itemize}

\end{enumerate}

\section{Zero entropy cases}

From the results of Theorems \ref{theo1}, \ref{theo2}, \ref{theo3}
and \ref{theo4} we can present the maps with zero algebraic entropy.

Looking at Theorem \ref{theo1} we see that the maps which satisfy
their hypothesis have not zero entropy. From Theorem \ref{theo2} we
find very interesting maps with zero entropy. This case is studied
in detail in the paper \cite{CZ0}, giving all the prescribed
invariant fibrations and also recognizing which of such a maps are
periodic and/or integrable.

Among the mappings satisfying the hypothesis of Theorem \ref{theo3}
we have the ones with $\alpha_1=0.$ After an affine change of
coordinates this maps can be written as
$$f(x,y)=\left(y,\frac{\beta_0+y}{\gamma_0+x}\right).$$
These are the maps which we deal when we want to study a linear
fractional recurrence of order two, and they are analized in paper
\cite{BK2}. When $\alpha_1\ne 0$ the only case with dynamical degree
equals one is when $\tilde{F}^2(A_2)=O_1,$ where $\tilde{F}$ is the
mapping induced by $F$ after blowing up the point $A_1.$ To find the
maps with this condition in principle we have two possibilities,
with or without collision of orbits. When $A_2\notin S_1$ then the
condition $\tilde{F}^2(A_2)=F^2(A_2)=O_1$ never is satisfied. It is
because $O_1\in S_0=T_2$ and the only points on $T_2$ which have
preimage by $F$ are $A_0$ and $A_1$ and we see that $A_0\ne O_1\ne
A_1.$ If $A_2\in S_1,$ that is if $\beta_0=\alpha_0,$ then
$$\tilde{F}^2(A_2)=\tilde{F}(\tilde{F}([1:0:-\alpha_1]))=\tilde{F}([1:\alpha_0-\alpha_1]_{E_1})=[0:\alpha_0-\alpha_1:1].$$
Hence, condition $\tilde{F}^2(A_2)=O_1=[0:1:-\alpha_1]$ is satisfied
for $\alpha_0=\frac{\alpha_1^2-1}{\alpha_1}=\beta_0$ and we get the
uniparametric family of mappings
$$f(x,y)=\left(\omega+\alpha_1 x+y,\frac{\omega+y}{x}\right)\,\,,\,\,\omega=\frac{\alpha_1^2-1}{\alpha_1}\,\,,\,\,\alpha_1\ne
0.$$ Since the corresponding sequence of degrees grows linearly, $f$
preserves a unique rational fibration. In fact,
$$V(x,y)=\frac{(1+\alpha_1 x)(\alpha_1+\alpha_1 x+y)}{x}$$
satisfies $V(f(x,y))=\alpha_1 V(x,y).$ When $\alpha_1^n=1$ for some
$n\in\N$ then defining $W(x,y)=V(x,y)^n$ we see that
$W(f(x,y))=W(x,y)$, that is $f$ is integrable. We observe that we
also know that these maps never are periodic maps as the degrees
grow linearly.

In a similar way, we find
$$f(x,y)=\left(\alpha_2 y,\frac{\beta_2
y}{-\alpha_2\beta_2+x+y}\right)$$ which satisfies the hypothesis of
Theorem \ref{theo4} with $\alpha_1=0$ and has the unique invariant
fibration
$$V(x,y)=\frac{(\beta_2-y)(\alpha_2\beta_2-x)}{y}$$
with the property $V(f(x,y))=-\alpha_2 V(x,y).$ As before when
$(-\alpha_2)^n=1,$ then $W(x,y)=V(x,y)^n$ is a first integral of
$f(x,y).$

Finally, if $f(x,y)$ satisfies the hypothesis of Theorem \ref{theo4}
with $\alpha_1\ne 0$ and has zero entropy, after an affine change of
coordinates can be written as
$$f(x,y)=\left(\alpha_0+\alpha_1 x,\frac{\beta_0+
y}{x+y}\right).$$ Clearly $V(x,y)=x$ is an invariant fibration
satisfying $V(f(x,y))=\alpha_0+\alpha_1 V(x,y).$ When
$\alpha_1^n=1\ne \alpha_1$ then calling $h(x):=\alpha_0+\alpha_1 x$
we have that $W(x,y)=x\,h(x)\,h^2(x)\,\ldots\,h^{n-1}(x)$ is a first
integral of $f(x,y).$ Also when $\alpha_1=1$ and $\alpha_0=0$ the
maps are integrable.

\end{document}